\documentclass[12pt]{article}
\usepackage{amssymb,amsfonts,amsmath,amsthm,tikz}
\usepackage{xcolor}
\usepackage[utf8]{inputenc}
\usepackage{bm}
\usepackage{mathtools}
\usepackage{soul}
\usepackage{enumerate}
\usepackage[english]{babel}
\usepackage{amsthm}
\usepackage{enumitem}
\usepackage{mathabx}
\usepackage{mathbbol}
\usepackage{flexisym}
\newcommand{\Z}{\mathbb{Z}}
\newcommand{\Q}{\mathbb{Q}}
\newcommand{\R}{\mathbb{R}}
\newcommand{\N}{\mathbb{N}}

\newcommand{\8}{\infty}

\newcommand{\I}{\mathcal{I}}
\newcommand{\qq}{\Q\times\Q}
\newcommand{\0}{\mathcal{O}}
\newcommand{\Ga}{\mathbb{\Gamma}}
\newcommand{\h}{\mathrm{Homeo}_{+}{(\R)}}
\theoremstyle{definition}
\newtheorem{question}{Question}
\usepackage[lite]{amsrefs}

\theoremstyle{definition}
\newtheorem{definition}{Definition}

\newtheorem{theorem}{Theorem}
\newtheorem{lemma}[theorem]{Lemma}
\newtheorem{proposition}{Proposition}
\newtheorem{corollary}[theorem]{Corollary}
\theoremstyle{remark}
\newtheorem{remark}{Remark}

\begin{document}
\theoremstyle{plain}
\newtheorem{conjecture}[theorem]{Conjecture}
\newtheorem{hypothesis}[theorem]{Hypothesis}
\newtheorem{condition}[theorem]{Condition}
\newtheorem{fact}[theorem]{Fact}
\newtheorem{problem}[theorem]{Problem}

\theoremstyle{definition}
\newtheorem{notation}[theorem]{Notation}

\theoremstyle{remark}

\title{On the topology of the space of bi-orderings of a free group on two generators}

\author{Serhii Dovhyi and Kyrylo Muliarchyk}

\date{}

\maketitle

\begin{abstract}
   Let $G$ be a group. We can topologize the spaces of left-orderings $LO(G)$ and bi-orderings $O(G)$ of $G$ with the product topology. These spaces may or may not have isolated points. It is known that $LO(F_n)$ has no isolated points, where $F_n$ is a free group on $n\geq 2$ generators. In this paper, we show that $O(F_n)$ has no isolated points as well, thereby resolving the second part of \cite[Conjecture 2.2]{Sikora}.
\end{abstract}

\section{Introduction}

Given a group $G$, a linear order $<$ is a left order if it is invariant under left multiplication, i.e., $x<y$ implies $zx<zy$ for all $x,y,z\in G$. A group that admits a left order is called  \textbf{left-orderable}.

\label{sec:1}

Elements that are bigger or smaller than the identity element of a group are called \textbf{positive} and \textbf{negative} respectively.

Another way to define left-orderability is as follows:

\begin{proposition}
\label{positive cone}
A group $G$ is left-orderable if and only if there exists a subset $P\subset G$ such that
\begin{enumerate}
\item $P\cdot P\subset P;$
\item for every $g\in G$, exactly one of $g=1$, $g\in P$ or $g^{-1}\in P$ holds.
\end{enumerate}
\end{proposition}

Such a subset $P$ is called a \textbf{positive cone}.

For a given order $<$ on a group $G$, the positive cone $P_<$ associated with this order is defined by $P_<\coloneqq \{g\in G\mid g>1\}$.
For a given positive cone $P\subset G$ the associated order $<_P$ is defined by $x<_P y$ if $x^{-1}y\in P$.

A group that admits a linear ordering which is invariant under both left and right multiplication is called \textbf{bi-orderable} or just \textbf{orderable}.

\begin{proposition}
\label{positive cone for bi-orderability}
A group $G$ is orderable if and only if it admits a subset $P$ satisfying the conditions $1$ and $2$ in Proposition \ref{positive cone}, and in addition, the condition

3. $gPg^{-1}\subset P$ for all $g\in G$.
\end{proposition}

\begin{proposition}
\label{family of left-orderable groups is closed under}
The family of orderable groups is closed under the following: taking subgroups, direct products, free products (first proved in \cite{Vin}), quotients by normal convex subgroups.

Moreover, orders on $G_1\times G_2$ and $G_1\ast G_2$ can be taken as the extensions of orders on $G_1$ and $G_2$. The order on $^{G}\!/\!_{N}$ could be defined as follows: $gN$ is positive in $^{G}\!/\!_{N}$ if $g$ is positive in $G$ and $g\notin N$.
\end{proposition}

In particular, free groups are orderable as free products of copies of the (orderable) group $\Z$.


Let $X$ be any set, and $\textbf{P}(X)$ be its power set.

The spaces $LO(G)\subset \textbf{P}(G)$ ($O(G)\subset \textbf{P}(G)$) of all left-invariant (bi-invariant) positive cones in $G$ was defined in \cite{Sikora}. As there is a one-to-one correspondence between
left-orderings (bi-orderings) of $G$ and left-invariant (bi-invariant) positive cones in $G$, it is natural to describe $LO(G)$ ($O(G)$) as the space of all left-orderings of $G$ (the space of all bi-orderings of $G$).

We follow \cite{CR15} in our exposition below.

The power set may be identified with the set of all functions $X\to \{0,1\}$, via the characteristic function $\chi_A:X\to \{0,1\}$ associated to a subset $A\subset X$. Give $\{0,1\}$ the discrete topology, and then one may consider $\textbf{P}(X)$ as a product of copies of $\{0,1\}$ indexed by the set $X$. The product topology is defined as the smallest topology on the power set $\textbf{P}(X)$ such that for each $x\in X$ the sets $U_x=\{A\subset X\mid x\in A\}$ and $U_x^{c}=\{A\subset X\mid x\notin A\}$ are open. A basis for the product topology can be obtained by taking finite intersections of various $U_x$ and $U_x^{c}$.

It is then natural to ask:
\begin{question}\label{q1}
How does the topological space $LO(G)$ ($O(G)$) look like for a given group $G$?
\end{question}

The following theorem was proved in \cite{Sikora}:

\begin{theorem}\label{first}
Let G be a countable orderable group. Then the space
$LO(G)$ is a compact totally disconnected Hausdorff metric space. The space $O(G)$ is a closed subset of $LO(G)$.
\end{theorem}

Let $<$ be a left-ordering of a group $G$, and let a finite chain of inequalities $g_0<g_1<\dots<g_n$ be given. Then the set of all left-orderings in which all these inequalities hold forms an open neighborhood of $<$ in $LO(G)$. The set of all such neighborhoods for all finite chains of inequalities is a local base for the topology of $LO(G)$ at the point $<$.

\begin{remark}
Instead of a chain of inequalities $g_0<g_1<\dots<g_n$ equivalently we may consider the sequence $x_1=g_0^{-1}g_1>1,\dots,x_n=g_{n-1}^{-1}g_n>1$, so $\{x_1,\dots,x_n\}\subset P_{<}$, for the positive cone $P_{<}$ associated with the order $<$.
\end{remark}

A left-ordering of $G$ is \textbf{isolated} in $LO(G)$ if it is the only left-ordering satisfying some finite chain of inequalities. Some groups $G$ have isolated points in $LO(G)$, while others do not.

Thus, by Theorem \ref{first} for a left-orderable (bi-orderable) countable group $G$, the space $LO(G)$ ($O(G)$) is homeomorphic to the Cantor set if and only if it has no isolated points.
We would, therefore, like to address the existence of isolated points in the space $LO(G)$ ($O(G)$) as the first step towards understanding
the structure of $LO(G)$ ($O(G)$).

It was established in \cite{Sikora} that a free abelian group $\Z^d$, $d\geq 2$ has no isolated orderings, and, therefore, the space $LO(\Z^d)=O(\Z^d)$ is isomorphic to the Cantor space.

The fundamental group of Klein bottle $K\cong \langle x,y\mid xyx^{-1}=y^{-1} \rangle$ has isolated orders \cite{CR15}. In fact, $K$ admits only finitely many (four) left-orders, all of them are isolated.

The Thompson's group $F$ has eight isolated bi-orders while $O(F)$ is uncountable \cite{navasrivas}.

Another important object  is the free group on two generators $F_2$. The following theorem was firstly proved in \cite{McCleary} 
\begin{theorem}\label{mainlo}
The space $LO(F_2)$ has no isolated points.
\end{theorem}

Later, Theorem \ref{mainlo} has been proved in many different ways. We are mostly interested in the idea presented in \cite{Navas}. 
A slightly modified strategy of this proof (see \cite{CR15}) is as follows:

\begin{proof}[Sketch]
Let $<$ be a left order on a free group $F_2$. 
\begin{enumerate}[label=Step \arabic*.,ref=\arabic*]
    \item\label{step1}  Embed a free group $F_2$ into a countable dense left-ordered group $G$.
    \item\label{step2}  Construct an order-preserving bijection $t:G\to\Q$.
    \item\label{step3} $F_2$ acts on $G$ by left multiplication. Using $t$ transform it to the action on $\Q$. Namely, for $g\in F_2$ let $\rho(g)(t(h)) = t(gh)$ where $h$ runs over $G$ and so $t(h)$ runs over $\Q$. Extend $\rho(g)$ to an action $\R\to\R$.
    \item\label{step4}   Let $a$ and $b$ be the generators of a free group $F_2$. Then $\rho(a)$ and $\rho(b)$ generates its copy in the group $\h$ of orientation-preserving homeomorphisms of $\R$. Let $\alpha,\beta\in \h$ be ''small" perturbations of $\rho(a),\rho(b)$ respectively. Consider a subgroup $H=\langle \alpha,\beta \rangle$ of $\h$ with the induced left order.
    \item\label{step5}   It remains to check that for an appropriate choice of $\alpha,\beta$, $H$ is a free group, and the new left order $\prec$ on it is ''close" but different from the initial left order $<$ on $F_2$.
\end{enumerate}
\end{proof}

The critical part of the above proof is the construction in steps \ref{step1}-\ref{step3}. More generally, each countable left-ordered group order-preserving embeds into $\h$ in a similar way. This embedding is called the \emph{dynamical realization of a left-ordered group}.

In this paper, we study the space $O(F_2)$ of orderings of a free group $F_2$ on two generators.

The main result of this paper is the following theorem:

\begin{theorem}\label{main}
The space $O(F_2)$ of orderings of a free group on two generators has no isolated points.
\end{theorem}

\begin{remark}
Although in this paper, we prove Theorem \ref{main} only for $O(F_2)$, all arguments can be appropriately adapted for all $O(F_n)$, $n>2$.
\end{remark}

We will follow the strategy from the above proof of Theorem \ref{mainlo} in our proof.

Similarly to step \ref{step1} from the above proof, we need to embed $F_2$ into a group with some density property. Our construction requires a stronger condition than a simple density. We will call this \emph{strong density}. It is discussed in Section \ref{sec:3}. Also, we prove that every countable bi-ordered group embeds into a countable strongly dense group.

Following steps \ref{step2}-\ref{step3}, in Section \ref{sec:4}, we construct two dynamical realizations of bi-ordered groups, one with action on $\qq$ equipped with lexicographic order, and another with action on $\R$. Equivalently, we construct embeddings of a countable ordered group $G$ into the group $\mathrm{Homeo}_{+}(\qq)$ of the order-preserving homeomorphisms of $\qq$ and $\h$. An important difference between bi-ordered and left-ordered cases is that $\mathrm{Homeo}_{+}(\qq)$ and $\h$ are left-ordered but not bi-ordered groups. Therefore not every faithful action on $\qq$ or $\R$ generates an order.

Finally, in Section \ref{sec:5} we show how to perturb a given order on $F_2$. We begin with a dynamical realization of $F_2$ (as an action on $\qq$). Then, we define a family of admissible changes of this action. Every member of this family will generate an order on $F_2$. To finish the proof of Theorem \ref{main}, we will choose a new order that sufficiently approximate the original order on $F_2$.

\textbf{Acknowledgements:}

Authors are grateful to Dr. Adam Clay for telling them about this conjecture, some known previously made attempts to solve it, reading drafts of this paper, and pointing out that the construction in Section \ref{sec:3} is an unrestricted wreath product.

\section{Further notation}

\label{sec:2}

Let $(G,<)$ be an ordered group, and $g,h\in G.$ We will use the following notations:
\begin{enumerate}
\item We denote the conjugation by $g^h\coloneqq hgh^{-1}$.

\item A subset $A\subset G$ is said to be convex if for any $f,h\in X$, $f<h$, every element $g\in G$ satisfying $f<g<h$ belong to $A$.

\item We denote by $\Gamma_g=\Gamma_g(G,<)$ the largest convex subgroup of $G$ that doesn't contain $g\in G\setminus\{1\}$. We denote the set of all 
such subgroups by $ \Ga=\Ga(G,<)$.


Similarly, we denote by $\bar{\Gamma}_g$ the smallest convex subgroup of $G$ containing $g\in G\setminus\{1\}$, and by $\bar{\Ga}$ the set of all 
such subgroups.

The group $G$ acts on $\Ga$ by conjugation. This action satisfies $\left(\Gamma_g\right)^h=\Gamma_{g^h}$.

The set $\Ga$ is naturally ordered by inclusion. The action of $G$ by conjugation preserves this order.

\item We will write $g\ll h$ when $\Gamma_g\subsetneq \Gamma_h$, or, equivalently, $g\in\Gamma_h$.

\item We say the elements of $G$ are equivalent if they define the same convex subgroup. 
Namely, $g\sim h$ if $\Gamma_g=\Gamma_h.$

\end{enumerate}

\section{Strongly dense groups}

\label{sec:3}

\begin{definition} 
\label{definition of strongly dense}
An ordered group $(G,<)$ is called \textbf{strongly dense} (with respect to the order $<$) if the following conditions are satisfied:
\begin{enumerate}
\item $\forall g_1,g_2\in G,g_1\ll g_2,\exists g_3\in G \hbox{ such that } g_1\ll g_3\ll g_2;$ 
\item $\forall g_1\in G\,\exists g_2,g_3\in G \hbox{ such that } g_2\ll g_1\ll g_3.$
\end{enumerate}
\end{definition}
In other words, the group $G$ is strongly dense if the correspondent set of convex subgroups $\Ga$, ordered by inclusion, is dense and doesn't contain the smallest and largest elements.

Replacement of the relation $\ll$ with the relation $<$ in Definition \ref{definition of strongly dense} above leads to the definition of a dense group. In the definition of a dense group, the second condition is omitted because the analogous condition is automatically satisfied.

Another approach to defining strongly dense groups is given below.

\begin{definition} Let $(H,<)$ be an ordered group.
\begin{enumerate}
\item A pair $(h^{\prime},h^{\prime\prime}),h^{\prime},h^{\prime\prime}\in H,h^{\prime}\ll h^{\prime\prime},$ is called a \textbf{gap} if there is no $h\in H$ such that $h^{\prime}\ll h\ll h^{\prime\prime}$.

Equivalently, $(h^{\prime},h^{\prime\prime})$ is a gap if $\Gamma_{h^{\prime\prime}}=\bar{\Gamma}_{h^{\prime}}$.

\item An element $h\in H$ is called a \textbf{peak} if there is no $h_1\in H$ such that $h\ll h_1$.

Equivalently, $h$ is a peak if $\bar{\Gamma}_h=H$.

\item An element $h\in H$ is called a \textbf{bottom} if there is no $h_1\in H\setminus\{1\}$ such that $h\gg h_1$.

Equivalently, $h$ is a bottom if $\Gamma_h=\{1\}$.

\end{enumerate}
\end{definition}

It is easy to see that an ordered group is strongly dense if and only if it contains no gaps, peaks, and bottoms.

The following theorem is the key result of this section.
\begin{theorem}\label{embeds strong} Any countable ordered group $F$ embeds in some countable strongly dense ordered group $G$.
\label{strong}
\end{theorem}

Our plan for proving this theorem is to eliminate peaks, gaps, and bottoms consecutively. By eliminating a peak, gap, or bottom of a countable group $H$, we understand the embedding $H$ in a countable ordered group $H_1$ without this peak, gap, or bottom. For example, if we want to eliminate a gap $(h^{\prime},h^{\prime\prime})$ in $H$, then we construct $H_1$ so there is $h_1\in H_1$ such that $h^{\prime}\ll h_1\ll h^{\prime\prime}$.

The next lemma states that we can always eliminate gaps, peaks, and bottoms.
\begin{lemma} 
For any countable ordered group $H$ and any peak, gap, or bottom in $H$, there is its countable extension $H_1$ without this peak, gap, or bottom.
\label{eliminate}
\end{lemma}

Firstly we prove that Lemma \ref{eliminate} implies Theorem \ref{strong}.
\begin{proof}[Proof of Theorem \ref{strong}]
Assume that we can eliminate gaps, peaks, and bottoms in any countable ordered group.

Since $F$ is a countable group, it contains at most countably many gaps, peaks, and bottoms. So we can enumerate all of them. Consider the chain $F=G_0=G_0^{(0)}<G_0^{(1)}<G_0^{(2)}<\dots$ of groups, constructed in the following way:
if the group $G_0^{(i)}$ contains the gap, peak, or bottom with number $i$ then we eliminate it, otherwise we set $G_0^{(i+1)}=G_0^{(i)}$.
The group $G_0^{(i)}$ is countable ordered and does not contain the $i^{\text{th}}$ gap, peak, or bottom of the group $F=G_0$.

Let $G_1=\bigcup\limits_{i}G_0^{(i)}$. The group $G_1$ is a countable ordered group without any gap, peak, and bottom of $F=G_0$, but possibly with new gaps, peaks, and bottoms.

Similarly, construct a new chain $G_1=G_1^{(0)}<G_1^{(1)}<\dots$. Get a countable ordered $G_2=\bigcup\limits_{i}G_1^{(i)}$ without any gaps, peaks, and bottoms of $G_1$. Then construct the chain $F=G_0<G_1<\dots$ of countable ordered groups, where $G_{i+1}$ does not contain any gap, peak, and bottom of $G_i$. Let $G=\bigcup\limits_{i}G_i$. Then $G$ is a countable ordered group.

Assume that $G$ is not strongly dense. Then it contains some gap, peak, or bottom. Let it be the gap $(g_1,g_2)$, where $g_1\in G_i$, $g_2\in G_j$. Then $(g_1,g_2)$ is a gap in $G_{\max\{i,j\}}$, so it has been eliminated during construction of $G_{\max\{i,j\}+1}$. This means $G>G_{\max\{i,j\}+1}$ does not contain the gap $(g_1,g_2)$. Contradiction. Similarly, the cases where $G$ contains a peak or a bottom are also impossible.

So $G$ is a countable strongly dense group.
\end{proof}

It remains to prove Lemma \ref{eliminate}.
\begin{proof}[Proof of Lemma \ref{eliminate}]

\textbf{How to eliminate a peak $h\in H$?}

The group $\mathbb{Z}\times H$, where $\mathbb{Z}=\langle z\rangle$ is an infinite cyclic group ordered lexicographically, does not have the peak $h\in H$, since $z\gg h,\forall h\in H$.

\textbf{How to eliminate a bottom $h\in H$?}

The group $H\times \mathbb{Z}$ ordered lexicographically does not have the bottom $h\in H$, since $z\ll h,\forall h\in H\setminus\{1\}$.

\textbf{How to eliminate a gap $(h^{\prime},h^{\prime\prime})$?}

To remove the gap $(h^{\prime},h^{\prime\prime})$, we want to construct a new ordered group $H_1>H$, with order $<$ on $H_1$ as an extension of the order $<$ on $H$, and there is an element $z\in H$ such that $h^{\prime}\ll z\ll h^{\prime\prime}$ in $H_1$. We will search for $H_1$ as a restricted wreath product $\mathbb{Z}\mathrm{ wr}_{\Omega}\,H$, where $\mathbb{Z}=\langle z\rangle$ is the infinite cyclic group. In fact, we add a new generator $z$ to $H$ and put it between $h^{\prime}$ and $h^{\prime\prime}$ to remove the gap.

\begin{remark}
The orderability of a restricted wreath product has been proved in \cite{neumann}. However, the order used in \cite{neumann} does not eliminate gaps. 
\end{remark}

We think of $H_1$ as a free product $H\ast \Z$ quotient by some relations.

When may an element $h\in H$ commute with $z$?

Elements $z$ and $h$ commute if and only if $z^h=z$.  Since the group $H_1$ supposed to be ordered, $h^{\prime}\ll z\ll h^{\prime\prime}$ implies $(h^{\prime})^h\ll z^h=z\ll (h^{\prime\prime})^h$. This is possible only if conjugation by $h$ preserves the classes $C_{h^{\prime}}$ and $C_{h^{\prime\prime}}$.

Let $M$ be a set of all those $h$, i.e.,
$M\coloneqq \{h\in H \mid (h^{\prime})^h\sim h^{\prime}\}$. Note that $M$ is a subgroup of $G$.

Conjugation by any $h\in H$ preserves the order $<$ and the class $C_{h^{\prime}}$. Since $(h^{\prime},h^{\prime\prime})$ is a gap, $C_{h^{\prime\prime}}$ is the smallest class larger then $C_{h^{\prime}}$.  Therefore, conjugation by $h$ also preserves the class $C_{h^{\prime\prime}}$.

This gives $M=\{h\in H\mid (h^{\prime\prime})^h\sim h^{\prime\prime}\}$.

Each element $h\in H\ast Z$ could be written in the form $h=h_0 \left(z^{\epsilon_1}\right)^{h_1} \dots \left(z^{\epsilon_n}\right)^{h_n}$, where $n\geq0$, $h_0,h_1,\dots,h_n\in H$,  $\epsilon_i\in\{\pm 1\}$, $i=1,\dots, n$.

\begin{remark}
Now we are ready to define $H_1$ as a restricted wreath product.
Let $\Omega$ be the set of left cosets of $M$ in $H$. 

The action of $H$ on $\Omega$ is multiplication on the left. Then $H_1\coloneqq \mathbb{Z}\mathrm{ wr}_{\Omega}\,H$.
\end{remark}

Recall that by definition $H_1=K\rtimes H$, where $K\coloneqq \bigoplus\limits_{w\in\Omega}\mathbb{Z_{\omega}}$ is the direct sum of copies of $\mathbb{Z_{\omega}}\coloneqq \mathbb{Z}$ indexed by the set $\Omega$.

Note that each $\omega\in\Omega$ has a form $\omega=h M$ for some $h\in H$.
We will use notation $\left(z^k\right)^{hM}$, $k\in\mathbb{Z}$, for elements of $\mathbb{Z}_{\omega}$.

Then $z^{h_1M}$ and $z^{h_2M}$ commute for all $h_1,h_2\in H$. In addition, $z^M$ and $m$ commute for all $m\in M$.

Taking into account the above notation, we can rewrite $H_1$ as 
\begin{equation*}
H_1=\{h_0(z^{k_1})^{h_1M}\dots(z^{k_n})^{h_n M}\mid n\geq0, h_0,\dots,h_n\in H,k_i\in\mathbb{Z}, i=1,\dots, n\}.
\end{equation*}

The multiplication $\bullet$ in $H_1$ comes from it being a restricted wreath product:
\begin{gather*}
\left(h_0^{(1)}(z^{k_1^{(1)}})^{h_1^{(1)}M}\dots(z^{k_{n_1}^{(1)}})^{h_{n_1}^{(1)}M}\right)\bullet \left(h_0^{(2)}(z^{k_1^{(2)}})^{h_1^{(2)}M}\dots(z^{k_{n_2}^{(2)}})^{h_{n_2}^{(2)}M}\right)\\
=h_0^{(1)}h_0^{(2)}(z^{k_1^{(1)}})^{h_0^{(2)}h_1^{(1)}M}
\dots(z^{k_{n_1}^{(1)}})^{h_0^{(2)}h_{n_1}^{(1)}M}(z^{k_1^{(2)}})^{h_1^{(2)}M}\dots(z^{k_{n_2}^{(2)}})^{h_{n_2}^{(2)}M}.
\end{gather*}

Then the inverse is defined as follows:
\begin{equation*}
(h_0(z^{k_1})^{h_1M}\dots(z^{k_{n}})^{h_{n}M})^{-1}=h_0^{-1}(z^{-k_1})^{h_0^{-1}h_1M}\dots(z^{-k_{n}})^{h_0^{-1}h_{n}M}.
\end{equation*}

Let us define the order on $H_1$.

Firstly, we extend the relation $\ll$ from $H$ to $H_1$ by the following rules:\\
1) $z^{h_1M}\ll z^{h_2M}$ if $(h^{\prime})^{h_1}\ll (h^{\prime})^{h_2}$ (or equivalently $(h^{\prime\prime})^{h_1}\ll (h^{\prime\prime})^{h_2}$) and $h_2^{-1}h_1\not\in M$; 

\noindent 2) $h\ll z^{h_1M}$ if  $h\ll (h^{\prime\prime})^{h_1}$, and $h\gg z^{h_1M}$ if $h\gg (h^{\prime})^{h_1}$.

\begin{remark}
\label{inmind}
From the definition of $M$ one can see that for every $h_1,h_2\in H$ exactly one of $z^{h_1M}\ll z^{h_2M}$, $z^{h_2M}\ll z^{h_1M}$ or $h_1M=h_2M$ holds. Therefore, all $z^{h_1 M}$, $h_1\in H$, are comparable with each other with respect to the relation $\ll$. Also every $z^{h_1 M}$, $h_1\in H$, is comparable with every $h\in H$.
\end{remark}

Keeping Remark \ref{inmind} in mind, we can define the positive cone $P_1$ of $H_1$ now as follows:

For $h=h_0(z^{k_1})^{h_1 M}\dots(z^{k_n})^{h_n M}\in H$ let $z^{h_i M}$ be the largest (with respect to the relation $\ll$) of $z^{h_j M}$, $j=1,\dots,n$. Now we say that $h\in P_1$ if either $h_0\gg z^{h_i M}$ and $h_0\in P$ (where $P$ is the positive cone of $H$), or $h_0\ll z^{h_i M}$ and $k_i>0$.

Checking the properties of Proposition \ref{positive cone for bi-orderability} for $P_1$ is straightforward.

Finally, note that $h^{\prime}\ll z^M\ll h^{\prime\prime}$, so the group $H_1$ does not contain the gap  $(h^{\prime},h^{\prime\prime})$.

\end{proof}

Let $F$ be an ordered group and $G$ be its ordered extension. We say a positive element $g\in G$ is small with respect to $F$ if $g<f$ for any positive $f\in F$.

\begin{lemma}\label{nosmall}
Let $F$ be an ordered group with no bottoms, and let $G>F$ be its strongly dense extension constructed as in Theorem \ref{embeds strong}. Then the group $G$ does not contain small with respect to $F$ elements.
\end{lemma}
\begin{proof}
Assume there is a small with respect to $F$ element in $G$.
Recall that $G=\bigcup\limits_{i}G_i$, with $F=G_0<G_1<G_2<\dots$. Let $G_i$ be the first group in the chain that contains a small with respect to $F$ element. Let $G_i=\bigcup\limits_{j}G_{i-1}^{(j)}$ with $G_{i-1}=G_{i-1}^{(0)}<G_{i-1}^{(1)}<G_{i-1}^{(2)}<\dots$, and let $G_{i-1}^{(j)}$ be the first group in the chain that contains a small with respect to $F$ element $g$. The group $G_{i-1}^{(j)}$ is constructed from the group $G_{i-1}^{(j-1)}$ by eliminating one of its gaps, peaks or bottoms as in Lemma \ref{eliminate}. 

Since every bottom of $G_{i-1}^{(j-1)}$ is small with respect to $F$, the group  $G_{i-1}^{(j-1)}$ has no bottoms.

Peak elimination clearly does not add small elements.

Let $G_{i-1}^{(j)}$ eliminated the gap $(g_1,g_2)$ of $G_{i-1}^{(j-1)}$.
Then, as was shown in Lemma \ref{eliminate}, $g\in G_{i-1}^{(j)}$ can be written as 
\begin{equation*}
    g=h_0 \left(z^{k_1}\right)^{h_1M}\dots \left(z^{k_n}\right)^{h_nM},
\end{equation*}
where $h_0,h_1,\dots, h_n\in  G_{i-1}^{(j-1)}$.

Let $h_0$ be the largest in $h_0, \left(z^{k_1}\right)^{h_1M},\dots, \left(z^{k_n}\right)^{h_nM}$. This means $h_0\gg (h_2)^{h_i}$, $i=1,\dots,n$. And $h_0>1$ as $g>1$.
Then
\begin{align*}
    g^2&=\left(h_0\left(z^{k_1}\right)^{h_1M}\dots \left(z^{k_n}\right)^{h_nM}\right)^2\\
    &=h_0^2 \left(\left(z^{k_1}\right)^{h_0h_1M}\dots \left(z^{k_n}\right)^{h_0h_nM}\right)\left(\left(z^{k_1}\right)^{h_1M}\dots \left(z^{k_n}\right)^{h_nM}\right)\\
    &=h_0 \left(h_0\left(z^{k_1}\right)^{h_0h_1M}\dots \left(z^{k_n}\right)^{h_0h_nM}\left(z^{k_1}\right)^{h_1M}\dots \left(z^{k_n}\right)^{h_nM}\right).
\end{align*}

Note that $h_0\gg (h_2)^{h_0h_i}$ since $h_0\gg (h_2)^{h_i}$, so 
\begin{equation*}
    h_0\left(z^{k_1}\right)^{h_0h_1M}\dots \left(z^{k_n}\right)^{h_0h_nM}\left(z^{k_1}\right)^{h_1M}\dots \left(z^{k_n}\right)^{h_nM}>0
\end{equation*}
and $g^2>h_0$.

Similarly, if $\left(z^{k_i}\right)^{h_iM}$ is the largest in $h_0, \left(z^{k_1}\right)^{h_1M},\dots, \left(z^{k_n}\right)^{h_nM}$ then $g^2>(z)^{h_iM}>g_1^{h_i}$.

In both cases $g$ is grater than some positive element $g^{\prime}\in G_{i-1}^{(j-1)}$. Since $G_{i-1}^{(j-1)}$ has no small with respect to $F$ elements, $g^{\prime}>f^{\prime}>1$ for some $f^{\prime}\in F$. Since $F$ has no bottoms, $f^{\prime}\gg f>1$ for some $f\in F$. So 
\begin{equation*}
    g^2>g^{\prime}>f^{\prime}>f^2
\end{equation*}
and $g>f$. Thus $g$ is not small with respect to $F$.

\end{proof}

\begin{corollary}\label{nosmallfree}
Let $F_2$ be an ordered free group, and let $G>F_2$ be its strongly dense extension as in Theorem \ref{embeds strong}. Then $G$ has no small with respect to $F_2$ elements.
\end{corollary}
\begin{proof}
This is true since $F_2$ with any order has no bottoms. Indeed,
let $f\in F_2$ be a bottom. Let $g\in F_2$ be any element that does not commute with $f$. We may assume $f>f^g>1$. Otherwise, we replace $g$ to $g^{-1}$. Then $f_g\in \bar{\Gamma}_f$. It follows from \cite[Theorem 2.3.1]{kopytovmedvedev} that the group $\Gamma_f$ is normal in $\bar{\Gamma}_f$ and the quotient $^{\bar{\Gamma}_f}\!/\!_{\Gamma_f}$ is abelian.
Since $f$ is a bottom, $\Gamma_f=\{1\}$ and $^{\bar{\Gamma}_f}\!/\!_{\Gamma_f}\cong\bar{\Gamma}_f$. But $\bar{\Gamma}_f$ is not abelian since it contains non-commutative elements $f$ and $f^g$. So $F_2$ has no bottoms.
\end{proof}

\section{Dynamical realization of bi-ordered groups}

\label{sec:4}

The dynamical realization of a left-ordered group $G$ rises from the action of $G$ on itself by left multiplication. If $G$ is a bi-ordered group, then it acts order-preserving on $\Ga$. This action corresponds to the action on $\Q$ as in the standard dynamical realization construction for left-ordered groups. However, this action is insufficient to construct a sort of dynamical realization. For instance, if $G$ is abelian, then the conjugation action always is trivial, so it provides no information about the order of $G$. Therefore we need a more complicated action. In this section, we will construct an action of $G$ on $\qq$ to prove the following theorem.

\begin{theorem}\label{dynre}
A countable group $G$ is bi-ordered if and only if it acts on $\Q\times \Q$ in the following way:
\begin{enumerate}[label=\Roman*]
    \item \label{tc1} the action order-preserving permutes layers $q\times \Q$, $q\in\Q$;
    \item \label{tc2} for any layer $q\times\Q$ the action by any element $g\in G$ on the second component of $q\times\Q$ is either trivial, increasing, or decreasing;
    \item \label{tc3} for each $g\in G\setminus\{1\}$ there is a layer $q_g\times \Q$ such that the fixed points under the action by $g$ are exactly the pairs $(q,p)\in\qq$ with $q>q_g$;
\end{enumerate}
\end{theorem}

\begin{proof}
We prove the ''if" part by showing that any group of all such actions on $\Q\times \Q$ is ordered. 

Let $F$ be the set of all such actions on $\qq$, and let $H\subset F$ be a group.

We say $g\in H$ is positive (negative) if it increases (decreases) the second component in the layer $q_g\times \Q$. Clearly, every nontrivial action is either positive or negative.

Consider two positive elements $g,h\in H$. Then $q_{gh}=\max\{q_g,q_h\}$ and the action of $gh$ increases the second component in the layer $q_{gh}\times\Q$. So, $gh$ is a positive element.

Consider a positive $g\in H$ and any $h\in H$. Then $q_{g^h}\times \Q=h(q_g\times\Q)$, and the action of $g^h$ on $q_{g^h}\times \Q$ is the conjugated action of $g$ on $q_g\times\Q$. So, it increases the second component. Then, $g^h$ is positive. 

Thus, $H$ is an ordered group.

It remains to prove the ''only if" part of the theorem.

Using Theorem \ref{embeds strong}, we may assume that the group $G$ is strongly dense. Then, by Cantor's back and forth argument \cite{cantorbf}, the set of convex subgroups $\Ga=\Ga(G)$ is order-preserving isomorphic $\Q$. We associate $\Ga$ with the first components in $\qq$. We plan to construct an ordered dense group $S$ and associate it with the second components in $\qq$. Thus, constructing the action on $\qq$ is equivalent to constructing the action on $\Ga\times S$.

We are looking for the action $\alpha$ of $G$ on $\Ga\times S$ of the following form. Let $\alpha_0$ be an order-preserving action of $G$ on $\Ga$, and let $\left\{\alpha_{\Gamma}\right\}_{\Gamma\in\Ga}$ be a collection of order-preserving actions of $G$ on $S$. Then the action $\alpha$ is given by
\begin{equation}\label{actionalpha}
    \alpha(g,(\Gamma,s))=(\alpha_0(g,\Gamma),\alpha_{\Gamma}(g,s)), \quad g\in G, \Gamma\in\Ga, s\in S.
\end{equation}

We require the actions $\alpha_0$, $\alpha_{\Gamma}$ to satisfy the following conditions
\begin{enumerate}
    \item\label{c1} for every $g\in G\setminus\{1\}$ there exists a convex subgroup $\Gamma(g)\in\Ga$ such that $\alpha_0(g,\Gamma)=\Gamma$ for all $\Gamma\geq\Gamma(g)$;
    
    \item \label{c2} each action $\alpha_{\Gamma}(g,\cdot)$ is either trivial, increasing, or decreasing;

    \item \label{c3} the action $\alpha_{\Gamma}(g,\cdot)$ is trivial if and only if $\Gamma>\Gamma(g)$ or $g=1$;
    
    \item\label{c4} the action $\alpha_{\Gamma(g)}(g,\cdot)$ is increasing for $g>1$ and decreasing for $g<1$;
    
    \item\label{c5} $\alpha_{h\cdot\Gamma}\bigl(g,\alpha_{\Gamma}(h,s)\bigr)=\alpha_{\Gamma}(gh,s)$, for all $g,h\in G$, $s\in S$, $\Gamma\in\Ga$.
\end{enumerate}

Condition \ref{c5} implies that $\alpha$ defined by \eqref{actionalpha} is a group action. Indeed,
\begin{align*}
    g\cdot(h\cdot(\Gamma,s)) &=g\cdot(h\cdot\Gamma,\alpha_{\Gamma}(h,s))=\biggl(g\cdot(h\cdot\Gamma),\alpha_{h\cdot\Gamma}\bigl(g,\alpha_{\Gamma}(h,s)\bigr)\biggr)\\
    &=\bigl((gh)\cdot\Gamma,\alpha_{\Gamma}(gh,s)\bigr)=(gh)\cdot(\Gamma,s).
\end{align*}

Conditions \ref{c1}-\ref{c3} mean that the action $\alpha$ satisfy the conditions \ref{tc1}-\ref{tc3} of the theorem. We have 
\begin{enumerate}[label=\Roman*]
    \item the action $\alpha$ order-preservingly permutes layers $\Gamma\times S$ according to the action $\alpha_0$;
    \item for a fixed layer $\Gamma\times S$ the action $\alpha_{\Gamma}(g,\cdot)$ is either trivial, increasing, or decreasing by condition \ref{c2};
    \item for each $g\in G\setminus\{1\}$ there is a layer $\Gamma(g)\times S$ such that the fixed points under the action $\alpha(g,\cdot)$ are exactly the pairs $(\Gamma,s)\in\Ga\times S$ with $\Gamma>\Gamma(g)$; 
\end{enumerate}

Condition \ref{c4} gives a characterization of the order

\begin{enumerate}[label=\Roman*,start=4]
    \item $g>1$ (respectfully, $g<1$) if the action $\alpha_{\Gamma(g)}(g,\cdot)$ is increasing (respectfully, decreasing).
\end{enumerate}

Next, we are going to build the group $S$ and such actions. 

For $\alpha_0$ we take the conjugation action $\alpha_0(g,\Gamma)=(\Gamma)^g$. We will take each action $\alpha_{\Gamma}(g,\cdot)$ to be a left multiplication by some $s_{g,\Gamma}\in S$. Then, with $\Gamma(g)=\Gamma_g$, the conditions \ref{c1} and \ref{c2} are automatically satisfied. The conditions \ref{c3}-\ref{c5} are transformed into
\begin{enumerate}[label=\arabic*\textprime.,ref=\arabic*\textprime,start=3]
    \item \label{c3prime} $s_{g,\Gamma}=1$ if and only if $g\in \Gamma$;
    \item \label{c4prime} $s_{g,\Gamma_g}>1$ for all $g>1$ and $s_{g,\Gamma_g}<1$ for all $g<1$;
    \item \label{c5prime}  $s_{g,\left(\Gamma\right)^h}\cdot s_{h,\Gamma}=s_{gh,\Gamma}$, for all $g,h\in G$, $\Gamma\in\Ga$.
\end{enumerate}

The action $\alpha_0$ splits $\Ga$ into orbits. Orbits are equivalence classes and they partition $\Ga=\0_1\cup \0_2\cup\dots$. For each orbit $\0_i$ we choose a representative $\Gamma_i\in\0_i$.
Let also $G_i=N_G(\Gamma_i)$ be the normalizer of $\Gamma_i$ in $G$. Clearly, $G_i<G$, and the order or $G$ induces the order on $G_i$. Then $\Gamma_i$ is a normal convex subgroup of $G_i$. Therefore the quotient group $H_i= ^{G_i}\!/\!_{\Gamma_i}$ is ordered with the order given by $h \Gamma_i\in H_i$ is positive if $h$ is positive in $G_i$ and $h\notin \Gamma_i$.  

For each $\Gamma\in\0_i$ we choose  $h_{\Gamma}\in G$ such that $\left(\Gamma\right)^{h_{\Gamma}}=\Gamma_i$. For a pair $(g,\Gamma)\in G\times \0_i$ consider the element $h_{g,\Gamma}=h_{(\Gamma)^{g}}gh_{\Gamma}^{-1}$. We have $\left(\Gamma_i\right)^{h_{\Gamma}^{-1}}=\Gamma$, $\left(\Gamma_i\right)^{gh_{\Gamma}^{-1}}=\left(\Gamma\right)^g$, and $\left(\Gamma_i\right)^{h_{g,\Gamma}}=\left(\left(\Gamma\right)^g\right)^{h_{(\Gamma)^{g}}}=\Gamma_i$. So $h_{g,\Gamma}\in G_i$.

Let $S=\left(\Asterisk_i H_i \right)\Asterisk F_{\infty}$, where $F_{\infty}=\langle f_{\Gamma}|\Gamma\in\Ga \rangle$ is an infinitely generated free group. The group $S$ is ordered as a free product of ordered groups; the order on $S$ extends the orders on each $H_i$ and is necessarily dense since the center of $S$ is trivial. Indeed, if the order on $S$ is not dense then there is the smallest positive element $g\in S$. But then for any $h\in S$, not commuting with $g$, either $1<g^h<g$ or $1<g^{h^{-1}}<g$ holds.

\begin{remark}\label{sprime}
The action $\alpha:G\times(\Ga\times S)\to \Ga\times S$ obviously extends to the action $\alpha^{\prime}:G\times(\Ga\times S^{\prime})\to \Ga\times S^{\prime}$, for any $S^{\prime}>S$. The action $\alpha^{\prime}$ is given by
\begin{equation}\label{actionalphaprime}
    \alpha^{\prime}(g,(\Gamma,s))=(\alpha_0(g,\Gamma),\alpha^{\prime}_{\Gamma}(g,s)),\quad g\in G, \Gamma\in\Ga, s\in S.
\end{equation}
Thus, in this construction, we can replace $S$ with any countable ordered dense group $S^{\prime}>S$.
\end{remark}

We set
\begin{equation}\label{s}
s_{g,\Gamma}\coloneqq  f_{(\Gamma)^g}\left( h_{g,\Gamma} \Gamma_i\right) f^{-1}_{\Gamma},\quad (g,\Gamma)\in G\times \0_i.
\end{equation}

It remains to verify that the elements $s_{g,\Gamma}$ satisfy the conditions \ref{c3prime}-\ref{c5prime}.
\begin{enumerate}[label=\arabic*\textprime.,start=3]
    \item \label{c3prime2} ($s_{g,\Gamma}=1$ if and only if $g\in \Gamma$).
    
    If $(\Gamma)^g\neq \Gamma$ then $s_{g,\Gamma}\neq 1$ since $f_{(\Gamma)^g}\neq f_{\Gamma}$. In this case $g\notin\Gamma$.
    Assume that $(\Gamma)^g=\Gamma$. Then $s_{g,\Gamma}=1$ is equivalent to $h_{g,\Gamma}=g^{h_{\Gamma}}\in\Gamma_i$ or $g\in(\Gamma_i)^{h^{-1}_{\Gamma}}=\Gamma$. So $s_{g,\Gamma}=1$ if and only if $g\in \Gamma$

    \item ($s_{g,\Gamma_g}>1$ for $g>1$). 
    
    Since $(\Gamma_g)^g=\Gamma_g$ we have $s_{g,\Gamma_g}=\left(h_{g,\Gamma_g}\Gamma_i\right)^{f_{\Gamma_g}}>1$ when $h_{g,\Gamma_g}\Gamma_i>1$ in $H_i$. 
    
    Note that $g\notin \Gamma_g$. So by \ref{c3prime2} $s_{g,\Gamma_g}\neq1$ and $h_{g,\Gamma_g}\Gamma_i\neq1$. 
    
    Then $h_{g,\Gamma_g}\Gamma_i>1$ in $H_i$ if $h_{g,\Gamma_g}>1$ in $G$. But $h_{g,\Gamma_g}=g^{h_{\Gamma_g}}>1$ as $g>1$.

    \item ($s_{g_1,\left(\Gamma\right)^{g_2}}\cdot s_{g_2,\Gamma}=s_{g_1g_2,\Gamma}$).
    
    We have
    \begin{align*}
    s_{g_1,\left(\Gamma\right)^{g_2}}\cdot s_{g_2,\Gamma}&=\left(f_{\left((\Gamma)^{g_2}\right)^{g_1}}\left( h_{g_1,(\Gamma)^{g_2}} \Gamma_i\right) f^{-1}_{(\Gamma)^{g_2}} \right)\cdot\left( f_{(\Gamma)^{g_2}}\left( h_{g_2,\Gamma} \Gamma_i\right) f^{-1}_{\Gamma}\right)\\
    &=f_{\left((\Gamma)^{g_2}\right)^{g_1}}\left( h_{g_1,(\Gamma)^{g_2}} \Gamma_i\right)\left( h_{g_2,\Gamma} \Gamma_i\right) f^{-1}_{\Gamma}\\
    &=f_{(\Gamma)^{g_1g_2}}\left( h_{g_1,(\Gamma)^{g_2}}\cdot h_{g_2,\Gamma} \Gamma_i\right) f^{-1}_{\Gamma}\\
    &=f_{(\Gamma)^{g_1g_2}}\left( h_{\left((\Gamma)^{g_2}\right)^{g_1}} g_1 h_{(\Gamma)^{g_2}}^{-1} \cdot h_{(\Gamma)^{g_2}}g_2 h_{\Gamma}^{-1} \Gamma_i\right) f^{-1}_{\Gamma}\\
    &=f_{(\Gamma)^{g_1g_2}}\left( h_{(\Gamma)^{g_1g_2}} (g_1g_2) h_{\Gamma}^{-1} \Gamma_i\right) f^{-1}_{\Gamma}\\
    &=s_{g_1g_2,\Gamma}.
    \end{align*}

\end{enumerate}

\end{proof}

Although the dynamical realization constructed in Theorem \ref{dynre} is sufficient for the purposes of this paper, it may be inconvenient for others. For example, the actions on $\qq$ are not continuous (if $\qq$ is granted with the standard topology).
Therefore we construct an alternative dynamical realization, with an action on $\R$, similarly to the standard dynamical realization of left-ordered groups.
We will show that every countable ordered group embeds into $\h$ granted with an order defined below.

Let a set $P\subset \h$ be given by
\begin{equation}\label{parposcon}
    P\coloneqq \{f\in\h\mid \sup\{x:f(x)>x\}>\sup\{x:f(x)<x\}\}.
\end{equation}
Here we consider the supremum of the empty set to be $-\infty$. 

It is easy to see that $P$ satisfies the following properties:
\begin{enumerate}
\item $P\cdot P\subset P$;
\item $f P f^{-1}\subset P$, for every $f\in \h$;
\item $P\cap P^{-1}=\emptyset$.
\end{enumerate}
So, $P$ defines a partial order on $\h$ given by $f<g$ when $f^{-1}g\in P$.

\begin{theorem}\label{dynre2}
Every countable ordered group $G$ embeds into $\h$ taken with the partial order define by the positive cone $P$ given by $\eqref{parposcon}$.
\end{theorem}
\begin{proof}
By Theorem \ref{dynre} there is a special action $\alpha$ of $G$ on $\qq$.

Consider $\qq$ with lexicographic order. Then it is an unbounded dense countable set, therefore, by Cantor's back and forth argument \cite{cantorbf}, it is order equivalent to $\Q$. Let $t:\qq\to\Q$ be an order preserving bijection. For each $g\in G$ we define the map $\rho(g):\R\to\R$ firstly on $\Q$ by the rule
\begin{equation*}
    \rho(g)(t(q,r))=t(\alpha(g,(q,r))),
\end{equation*}
then we extend it continuously to the action on $\R$.

Then $g\mapsto \rho(g)$ is the embedding of $G$ into $\h$.
\end{proof}

\begin{remark}
The order on $G$ may be considered to be a left order. Then the standard dynamical left-ordered realization $\rho_l:G\to \h$ of the bi-ordered group $G$ proves Theorem \ref{dynre2}. In fact, for every $g>1$ in $G$, the graph of $\rho_l(g)$ is above the line $y=x$. Therefore, we have $\sup\{x:\rho_l(g)(x)>x\}=+\infty$ and $\sup\{x:\rho_l(g)(x)<x\}=-\infty$.

However, the properties of $\rho$ and $\rho_l$ are quite different.
\end{remark}

Consider a layer $q\times\Q$. It corresponds to rational points on the interval $I_q\coloneqq (\inf\{t(q,r):r\in\Q\},\sup\{t(q,r):r\in\Q\})$. Let $\I\coloneqq \{I_q:q\in\Q\}$ be the set of all such intervals. Then it is easy to see that
\begin{enumerate}[label=(\alph*)]
    \item for any $g\in G$ and $I\in\I$ we have $\rho(g)(I)\in\I$;
    \item for any $g\in G$ and $I\in\I$, $\rho(g)(x)\geq x$ or $\rho(g)(x)\leq x$ holds simultaneously for $x\in I$;
    \item for any $g\neq 1$ there is an interval $I_g=I_{q_g}=(p,q)\in\I$, such that $\rho(g)(x)=x$ for $x>q$;
    \item for any $g\neq 1$ and any $I\leq I_g$ there is $x\in I$ such that $\rho(g)(x)\neq x$;
    \item for any $g>1$ and any $x\in I_g$, $g(x)\geq x$.
\end{enumerate}

\section{Changing the order}

\label{sec:5}

In this section we prove our main result Theorem \ref{main}. That is,
any given order $<$ on $F_2=\langle a,b \rangle$ is not isolated in $O(F_2)$. For any collection of positive elements $g_1,\dots,g_n>1$ we need to construct a new order $\prec\neq<$, such that $g_1,\dots,g_n\succ1$. 

Let $F_2$ embeds into a strongly dense countable group $G$, and the order $<$ on $G$ extends the order $<$ on $F_2$. Everywhere in this section, we let $\Gamma_g=\Gamma_g(G,<)$ and $\Ga=\Ga(G,<)$ are considered with respect to the group $G$ and order $<$ on it. We consider the dynamical realization of $G$.

Recall that the dynamical realization was constructed in Section \ref{sec:4} using the action $\alpha$ of $G>F_2$ on $\Ga\times S$. By Remark \ref{sprime} we may replace the group $S$ with any countable ordered dense supgroup $S^{\prime}>S$ in this construction. We consider $S^{\prime}=S\ast F_{\8}^{(a)}\ast F_{\8}^{(b)}$, where $F^{(a)}_{\8}=\langle f_{a,\Gamma}\mid \Gamma\in\Ga\rangle$, $F^{(b)}_{\8}=\langle f_{b,\Gamma}\mid \Gamma\in\Ga \rangle$ are infinitely generated free groups. 
The group $S^{\prime}$ is ordered since it is a free product of ordered groups, and we choose an order on $S^{\prime}$ to extend the order on $S$. Moreover, this order is dense since the center of $S^{\prime}$ is trivial.
We will construct a family of alternative actions of $F_2$ on $\Ga\times S^{\prime}$, where $S^{\prime}$ is a countable ordered extension of $S$.

The action of $\alpha^{\prime}$ was constructed through the (conjugation) action $\alpha_0$ of $G$ on $\Ga$ and a collection of (left multiplications by $s_{g,\Gamma}$) actions $\left\{\alpha^{\prime}_{\Gamma}\right\}_{\Gamma\in\Ga}$ on $S^{\prime}$.
Namely,
\begin{equation*}
g\cdot(\Gamma,s)=\left((\Gamma)^g,s_{g,\Gamma} s\right),\quad g\in F_2,\Gamma\in\Ga,s\in S^{\prime}.
\end{equation*}

For simplicity, we will write $\alpha$ instead of $\alpha^{\prime}$.

Similarly to the construction in Theorem \ref{dynre} we build alternative actions $F_2$ on $\Ga\times S^{\prime}$. 
Similarly to $\alpha$, the new action $\beta$ is defined by an action $\beta_0$ of $G$ on $\Ga$, and a collection of actions $\left\{\beta_{\Gamma}\right\}_{\Gamma\in\Ga}$ on $S^{\prime}$.
Each action $\beta_{\Gamma}(g,\cdot)$ is the left multiplication by $s_{g,\Gamma}^{\prime}\in S^{\prime}$.

To construct the action $\beta$ of the group $F_2=\langle a,b\rangle$ we need to define the actions of its generators $a,b$.
We begin with defining the action $\beta_0$ on the set of convex subgroups $\Ga$.

Fix some convex subgroup $\Gamma_0\in\Ga$.
We define
\begin{equation*}
    \beta_0(c,\Gamma)=c\cdot\Gamma=\left(\Gamma\right)^c,\quad \Gamma\geq\Gamma_0\cup\left(\Gamma_0\right)^{c^{-1}},c\in\{a,a^{-1},b,b^{-1}\}.
\end{equation*}

In other words, $c\cdot\Gamma=(\Gamma)^c$ for $\Gamma_0\leq \left(\Gamma\right)^c\cap\Gamma$. We extend the action $\beta_0$ so that $\beta_0(a,\cdot)$ and $\beta_0(b,\cdot)$ are order preserving bijections $\Ga\to\Ga$. We can always extend these actions using back and forth argument. Moreover, for any $\Gamma<\Gamma_0\cup\left(\Gamma_0\right)^{c^{-1}}$ we can choose $c\cdot \Gamma$ to be any convex subgroup $<\left(\Gamma_0\cup\left(\Gamma_0\right)^{c^{-1}}\right)^c=\Gamma_0\cup\left(\Gamma_0\right)^c$, $c\in\{a,a^{-1},b,b^{-1}\}$.

Everywhere below $g\cdot \Gamma$ means $\beta_0(g,\Gamma)$ and $g\cdot(\Gamma,s)$ means $\beta(g,(\Gamma,s))$.

Let us now define the actions $\beta_{\Gamma}$. Recall that the action $\beta_{\Gamma}$ is a multiplication by $s^{\prime}_{g,\Gamma}\in S^{\prime}$. For $c\in\{a,b\}$ we define
\begin{equation*}
    s^{\prime}_{c,\Gamma}=\left\{\begin{array}{cc}
     s_{c,\Gamma},  & \; \Gamma> \Gamma_0\cup\left(\Gamma_0\right)^{c^{-1}}; \\
     f_{c,\Gamma},  & \;\Gamma\leq \Gamma_0\cup\left(\Gamma_0\right)^{c^{-1}}
    \end{array}\right.
\end{equation*}
and
\begin{equation*}
    s^{\prime}_{c^{-1},\Gamma}=\left(s^{\prime}_{c,c^{-1}\cdot\Gamma}\right)^{-1}=\left\{\begin{array}{cc}
     s^{-1}_{c,c^{-1}\cdot\Gamma},  & \; \Gamma> \Gamma_0\cup\left(\Gamma_0\right)^c; \\
     f^{-1}_{c,c^{-1}\cdot\Gamma},  & \;\Gamma\leq \Gamma_0\cup\left(\Gamma_0\right)^c.
    \end{array}\right.
\end{equation*}

We also denote $f_{c^{-1},\Gamma}\coloneqq f^{-1}_{c,c^{-1}\cdot \Gamma}$. Then for $c\in\{a,a^{-1},b,b^{-1}\}$
\begin{equation}\label{sorf}
     s^{\prime}_{c,\Gamma}=\left\{\begin{array}{cc}
     s_{c,\Gamma},  & \; \Gamma> \Gamma_0\cup\left(\Gamma_0\right)^{c^{-1}}; \\
     f_{c,\Gamma},  & \;\Gamma\leq \Gamma_0\cup\left(\Gamma_0\right)^{c^{-1}}
    \end{array}\right.
\end{equation}

In other words, the new actions $\beta_{\Gamma}(c,\cdot)$ are equal to the old actions $\alpha_{\Gamma}(c,\cdot)$ for sufficiently large $\Gamma$'s, and are multiplications by $f_{c,\Gamma}\in F^{(a)}_{\infty}\ast F^{(b)}_{\infty}<S^{\prime}$ for sufficiently small $\Gamma$'s.

For $g=c_n\dots c_1$, $c_i\in\{a,a^{-1},b,b^{-1}\}$, $i=1,\dots n$, we have
\begin{align*}
    g\cdot(\Gamma,s)&=c_n\cdot\dots\cdot c_2\cdot c_1\cdot(\Gamma,s)\\
    &=c_n\cdot\dots\cdot c_2\cdot(c_1\cdot\Gamma,s^{\prime}_{c_1}s)=\dots\\
    &=(c_n\cdot\dots \cdot c_1\cdot \Gamma,s^{\prime}_{c_n}\dots s^{\prime}_{c_1}s).
\end{align*}

We denote 
\begin{equation}\label{prodprime}
    s^{\prime}_{g,\Gamma}=s^{\prime}_{c_n,\Gamma_n}\dots s^{\prime}_{c_1,\Gamma_1}.
\end{equation}
Then we have
\begin{equation}\label{eqsprime}
g\cdot(\Gamma,s)=(g\cdot \Gamma,s^{\prime}_{g,\Gamma}s), \quad g\in F_2,\Gamma\in\Ga,s\in S.
\end{equation}

\begin{remark}
The constructed action $\beta$ depends on the choice of the order on $S^{\prime}$, the convex subgroup $\Gamma_0$, and the actions of $a$ and $b$ on the small convex subgroups.
\end{remark}

Since $\Ga$ and $S^{\prime}$ are countable dense sets, we can see $\beta$ as an action on $\qq$ instead of $\Ga\times S^{\prime}$.

\begin{theorem}\label{gd}
For any choice of the order on $S^{\prime}$, $\Gamma_0$, and any actions of $a$, $b$ on $\Ga$, the constructed as above action $\beta$ satisfies the conditions \ref{tc1}-\ref{tc3} from Theorem \ref{dynre}, and, therefore, $\beta$ defines some order $\prec$ on $F_2$. 
\end{theorem}

\begin{proof}
Recall that by \eqref{eqsprime}, for any $\Gamma\in \Ga$, we have $g\cdot(\Gamma,s)=(g\cdot \Gamma,s^{\prime}_{g,\Gamma}s)$.

Similarly to the proof of Theorem \ref{dynre} we will show that for any $g\in F_2\setminus\{1\}$ there is a convex subgroup $\Gamma_g^{\prime}$ such that $s_{g,\Gamma}^{\prime}\neq 1$ if and only if $\Gamma\leq \Gamma_g^{\prime}$, and $g\cdot\Gamma=\Gamma$ for $\Gamma>\Gamma_g^{\prime}$. Then the action $\beta$ defines the order $\prec$ on $F_2$, given by $g\succ 1$ when $s_{g,\Gamma_g^{\prime}}^{\prime}>1$ in $S^{\prime}$.

Recall that the initial order $<$ is given by $g>1$ when $s_{g,\Gamma_g}>1$ in $S^{\prime}$.

Consider $g=c_n\dots c_1\neq 1$, $c_i\in\{a,a^{-1},b,b^{-1}\}$.
Recall that then by $\eqref{prodprime}$
\begin{equation*}
     s^{\prime}_{g,\Gamma}=s^{\prime}_{c_n,\Gamma_n}\dots s^{\prime}_{c_1,\Gamma_1}.
\end{equation*}
In this product each $s^{\prime}_{c_i,\Gamma_i}$ is either $s_{c_i,\Gamma_i}$ or $f_{c_i,\Gamma_i}$. 
Let
\begin{equation*}
    g_i=c_{i-1}\dots c_1
\end{equation*}
be the word containing the last $i-1$ letters of the word $g$. In particular, $g_1=1$ is the trivial word. Note that $\Gamma_i=g_i\cdot \Gamma$. Recall that by \eqref{sorf}
\begin{equation*}
s_{c_i,\Gamma_i}^{\prime}=s_{c_i,\Gamma_i} \iff \Gamma_i>\Gamma_0\cup(\Gamma_0)^{c_i^{-1}} \iff \Gamma>g_i^{-1}\cdot\left(\Gamma_0\cup(\Gamma_0)^{c_i^{-1}}\right).
\end{equation*}

Let 
\begin{equation*}
\Gamma^{(i)}_g=g_i^{-1}\cdot\left(\Gamma_0\cup(\Gamma_0)^{c_i^{-1}}\right).
\end{equation*}
Then $\Gamma=\Gamma_g^{(i)}$ is the largest convex subgroup such that the $i^{\text{th}}$ from the right letter $s^{\prime}_{c_i,\Gamma_i}$ in the word $s^{\prime}_{g,\Gamma}$ is an $f$-letter (i.e. equal to $f_{c_i,\Gamma_i}$). 

Also, for a set of indexes  
\begin{equation*}
I=\{i_1,i_2,\dots,i_k\mid 1\leq i_1<i_2<\dots<i_k\leq n\}
\end{equation*}
we denote
\begin{equation*}
    g_I\coloneqq c_{i_k}\dots c_{i_1}
\end{equation*}
and
\begin{equation*}
\Gamma^{(I)}_g\coloneqq g_{i_1}^{-1}\cdot  \Gamma_{g_I}.
\end{equation*}

We claim that for $g\in F_2$ either $g=1$ or the convex subgroup $\Gamma_g^{\prime}$ is one of the $\Gamma_g^{(i)}$'s or $\Gamma_g^{(I)}$'s. Namely, the largest of them such that $s^{\prime}_{g,\Gamma}\neq1$. We prove this statement by induction on the length of $g$. The base case $g=1$ holds trivially. 

Let $g\neq 1$ be a reduced word and assume that for any shorter word $h\neq 1$ there exists a convex subgroup $\Gamma_h^{\prime}\in\{\Gamma^{(i)}_h,\Gamma_h^{(I)}\}$ such that $s_{h,\Gamma}^{\prime}\neq 1$ if and only if $\Gamma\leq \Gamma_h^{\prime}$, and $h\cdot\Gamma=\Gamma$ for $\Gamma>\Gamma_h^{\prime}$.

Let $\Gamma_{g,1}^{\prime}<\Gamma^{\prime}_{g,2}<\dots<\Gamma^{\prime}_{g,N}$ be the convex subgroups $\Gamma_g^{(i)}$'s and $\Gamma_g^{(I)}$'s ordered by inclusion.

Consider the case $\Gamma\leq\Gamma_{g,1}^{\prime}$.

Then $\Gamma\leq \Gamma_g^{(i)}$, $1\leq i\leq n$, and therefore 
\begin{equation*}
s^{\prime}_{g,\Gamma}= s^{\prime}_{c_n,\Gamma_n} \dots s^{\prime}_{c_1,\Gamma_1} =f_{c_n,\Gamma_n} \dots f_{c_1,\Gamma_1} 
\end{equation*}
is an $f$-word, i.e. every letter of it is an $f$-letter (an element of $F^{(a)}_{\infty}\ast F^{(a)}_{\infty}$). So $s^{\prime}_{g,\Gamma}\neq 1$ as a nontrivial reduced word in the free group $F_{\infty}^{(a)}\ast F_{\infty}^{(b)}$. In particular, $s_{g,\Gamma_{g,1}}^{\prime}\neq 1$.

Consider the case $\Gamma>\Gamma^{\prime}_{g,N}$. Then $\Gamma> \Gamma_g^{(i)}$, $1\leq i\leq n$, and therefore 
\begin{equation*}
s^{\prime}_{g,\Gamma}=s^{\prime}_{c_n,\Gamma_n}\dots s^{\prime}_{c_1,\Gamma_1}=s_{c_n,\Gamma_n}\dots s_{c_1,\Gamma_1}=s_{g,\Gamma}.
\end{equation*}
Also $\Gamma>\Gamma_g^{(I)}$ with $I=\{1,2,\dots,n\}$. So $s^{\prime}_{g,\Gamma}=s_{g,\Gamma}=1$.

It remains to consider $\Gamma\in(\Gamma^{\prime}_{g,1},\Gamma^{\prime}_{g,N}]$. If $s_{g,\Gamma}^{\prime}\neq 1$ for all $\Gamma\in(\Gamma^{\prime}_{g,1},\Gamma^{\prime}_{g,N}]$ then theorem is proven with $\Gamma_{g}^{\prime}=\Gamma^{\prime}_{g,N}$.
Let $\Gamma\in(\Gamma_{g,k}^{\prime},\Gamma_{g,k+1}^{\prime}]$ and $s^{\prime}_{g,\Gamma}=1$. We choose the smallest $k$ for which such $\Gamma$ exists. So $s_{g,\Gamma}^{\prime}\neq 1$ for all $\Gamma\in(\Gamma^{\prime}_{g,1},\Gamma^{\prime}_{g,k}]$.
We will show $\Gamma^{\prime}_g=\Gamma^{\prime}_{g,k}$.

In the word $s^{\prime}_{g,\Gamma}$ we combine consecutive $s$-letters and $f$-letters. Write
$s^{\prime}_{g,\Gamma}=s_1 f_1 s_2 f_2 \dots s_l f_l s_{l+1}$, where $s_i\in S$ and $f_j\in F^{(a)}_{\infty}\ast F^{(b)}_{\infty}$. All of $s_i$'s and $f_j$'s are nonempty words accept possibly $s_1$ and $s_{l+1}$. Note that all $f_j$'s are nontrivial as elements of the free group $F^{(a)}_{\infty}\ast F^{(b)}_{\infty}$. So $s_{g,\Gamma}^{\prime}=1$ is possible only if some $s_i=1$. Let 
\begin{equation*}
s_i=s^{\prime}_{c_r,\Gamma_r}\dots s^{\prime}_{c_t,\Gamma_t}=s_{c_r,\Gamma_r}\dots s_{c_t,\Gamma_t}=1.
\end{equation*}

Here $s_i$ is an $s$-subword so all its letter are $s$-letters. This means $\Gamma>\Gamma_g^{(j)}$, $j=r,\dots,t$. Also $s_{c_r,\Gamma_r}\dots s_{c_t,\Gamma_t}=1$ if and only if $\Gamma>\Gamma_g^{(I)}$ with $I=\{r,\dots,t\}$. Since $\Gamma\in(\Gamma^{\prime}_{g,k},\Gamma^{\prime}_{g,k+1}]$ we have $\Gamma^{\prime}_{g,k}\geq \min\{\Gamma^{(I)}_g,\Gamma^{(r)}_g,\dots,\Gamma^{(t)}_g\}$.

We write 
\begin{equation*}
    g=uvw,
\end{equation*}
where 
\begin{equation*}
   u=c_n\dots c_{r+1},\quad
    v=c_r\dots c_t,\quad\text{and}\quad
    w=c_{t-1}\dots c_1.
\end{equation*}

Then we have
\begin{equation}\label{subwordcor}
    s_i=s^{\prime}_{v,\Gamma_t}=s^{\prime}_{v,w\cdot\Gamma}=s_{v,w\cdot\Gamma}=1.
\end{equation}

Consider a convex subgroup $\widetilde{\Gamma}>\Gamma^{\prime}_{g,k}$ and let $\widetilde{s}_i=s^{\prime}_{c_r,\widetilde{\Gamma}_r}\dots s^{\prime}_{c_t,\widetilde{\Gamma}_t}$ be the subword of $s^{\prime}_{g,\widetilde{\Gamma}}$ whose letters are located at the same positions in $s^{\prime}_{g,\widetilde{\Gamma}}$ as the letters of $s_i$ in $s^{\prime}_{g,\Gamma}$.
Here $\widetilde{\Gamma}_1=\widetilde{\Gamma}$, $\widetilde{\Gamma}_{i+1}=c_i\cdot\widetilde{\Gamma}_i$.
Then, since $\widetilde{\Gamma}> \min\{\Gamma^{(r)}_g,\dots,\Gamma^{(t)}_g\}$, $\widetilde{s}_i$ is an $s$-word, and, since $\widetilde{\Gamma}>\Gamma_g^{(I)}$, $\widetilde{s}_i=1$. Similarly to \eqref{subwordcor} we have

\begin{equation*}
    \widetilde{s}_i=s^{\prime}_{v,\widetilde{\Gamma}_t}=s^{\prime}_{v,w\cdot\widetilde{\Gamma}}=s_{v,w\cdot\widetilde{\Gamma}}=1.
\end{equation*}

Note that we also have
\begin{equation*}
    v\cdot\left(w\cdot \widetilde{\Gamma}\right)=\left(w\cdot\widetilde{\Gamma}\right)^v=w\cdot\widetilde{\Gamma}.
\end{equation*}

Therefore 
\begin{equation}\label{fix}
    v\cdot\left(w\cdot \widetilde{\Gamma},s\right)=\left(w\cdot \widetilde{\Gamma},\widetilde{s}_is\right)=\left(w\cdot \widetilde{\Gamma},s\right)
\end{equation}
whenever $\widetilde{\Gamma}>\Gamma^{\prime}_{g,k}$.

Let $h$ be obtained from $g$ by removing the subword $v$. We have
\begin{equation*}
    h=(c_n\dots c_{r+1}) (c_{t-1}\dots c_1)=uw.
\end{equation*}

Note that the word $h$ is shorter than $g$. We claim that for $\widetilde{\Gamma}>\Gamma^{\prime}_{g,k}$ we have
\begin{equation*}
    h\cdot(\widetilde{\Gamma},s)= g\cdot(\widetilde{\Gamma},s). 
\end{equation*}

Consider $h^{-1}g=(uw)^{-1}uvw=w^{-1}vw$.
We have
\begin{equation*}
    h^{-1}g\cdot(\widetilde{\Gamma},s)=w^{-1}vw\cdot(\widetilde{\Gamma},s)=w^{-1}v\cdot\left(w\cdot \widetilde{\Gamma},s_{w,\widetilde{\Gamma}}s\right).
\end{equation*}

Using \eqref{fix} we get
\begin{equation*}
    h^{-1}g\cdot(\widetilde{\Gamma},s)=w^{-1}\cdot\left(w\cdot \widetilde{\Gamma},s_{w,\widetilde{\Gamma}}s\right)=(\widetilde{\Gamma},s).
\end{equation*}

Thus,
\begin{equation*}
    g\cdot(\widetilde{\Gamma},s)=h\cdot\left(h^{-1}g\cdot(\widetilde{\Gamma},s)\right)=h\cdot(\widetilde{\Gamma},s).
\end{equation*}

If $h=1$ we immediately have $s_{g,\widetilde{\Gamma}}^{\prime}=s_{1,\widetilde{\Gamma}}^{\prime}=1$ and $g\cdot\widetilde{\Gamma}=1\cdot\widetilde{\Gamma}=\widetilde{\Gamma}$ for all $\widetilde{\Gamma}>\Gamma_{g,k}^{\prime}$. Recall that $s_{g,{\Gamma}}^{\prime}\neq 1$ for $\Gamma\leq \Gamma_{g,k}^{\prime}$.
Therefore $\Gamma^{\prime}_g=\Gamma^{\prime}_{g,k}$.

Let $h\neq 1$. Then, after applying the inductive assumption, we obtain
$s_{g,\widetilde{\Gamma}}^{\prime}=s_{h,\widetilde{\Gamma}}^{\prime}=1$ and $g\cdot\widetilde{\Gamma}=h\cdot\widetilde{\Gamma}=\widetilde{\Gamma}$ for all $\widetilde{\Gamma}>\Gamma_{g,k}^{\prime}\cup\Gamma_{h}^{\prime}$.
Here $\Gamma_{h}^{\prime}=\Gamma_{h,m}^{\prime}$ for some $m$. Note that every $\Gamma_{h,i}^{\prime}$ is one of the $\Gamma_{g,j}^{\prime}$'s. And $s_{h,\Gamma}^{\prime}=1$ implies $\Gamma_{h}^{\prime}<\Gamma_{g,k+1}^{\prime}$. So $\Gamma_h^{\prime}\leq \Gamma_{g,k}^{\prime}$. Then  $\Gamma_{g,k}^{\prime}\cup\Gamma_{h}^{\prime}=\Gamma_{g,k}^{\prime}$. Again, $s_{g,\widetilde{\Gamma}}^{\prime}=1$ and $g\cdot\widetilde{\Gamma}=\widetilde{\Gamma}$ for all $\widetilde{\Gamma}>\Gamma_{g,k}^{\prime}$, so $\Gamma^{\prime}_g=\Gamma^{\prime}_{g,k}$.

\end{proof}

Now we can prove the main result of this paper.

\begin{proof}[Proof of Theorem \ref{main}]
We need to show that a given order $<$ on $F_2$ is not isolated. That is for any sequence of positive elements $g_1,\dots,g_k$ there is another order $\prec\neq<$ satisfying $g_1,\dots, g_k\succ1$.

We consider the order $\prec$ to be associated with the action $\beta$ as in Theorem \ref{gd}. We need to choose a convex subgroup $\Gamma_0\in\Ga$, an order on $S^{\prime}$, and an action $\beta_0$ of $F_2$ on $\Ga$. By choosing sufficiently small $\Gamma_0$, we guarantee $g_1,\dots, g_k\succ1$, and by choosing appropriate order on $S^{\prime}$ and action $\beta$, we make the new order $\prec$ different from the old order $<$.

For $g_i=c_{n_i}^{(i)}\dots c_{1}^{(i)}$ $i=1,\dots,k$, by $\eqref{prodprime}$ we have
\begin{equation*}
    s^{\prime}_{g_i,\Gamma_{g_i}}=s^{\prime}_{c_{n_i}^{(i)},\Gamma_{n_i}^{(i)}}\dots s^{\prime}_{c_1^{(i)},\Gamma_1^{(i)}},
\end{equation*}
$\Gamma_1^{(i)}=\Gamma_{g_i}$, $\Gamma_{j+1}^{(i)}=c_{j}^{(i)}\cdot\Gamma_j^{(i)}$, $j=1,\dots, n_i-1$. Then, if $\Gamma_0<\bigcap\limits_{j=1}^{n_i}\Gamma_j^{(i)}$, we have $\Gamma_{j+1}^{(i)}=c_{j}^{(i)}\cdot\Gamma_j^{(i)}=\left(\Gamma_j^{(i)}\right)^{c_j{(i)}}$ and $s^{\prime}_{c_j^{(i)}}=s_{c_j^{(i)}}$. So
\begin{equation*}
    s^{\prime}_{g_i,\Gamma_{g_i}}=s^{\prime}_{c_{n_i}^{(i)},\Gamma_{n_i}^{(i)}}\dots s^{\prime}_{c_1^{(i)},\Gamma_1^{(i)}}=s_{c_{n_i}^{(i)},\Gamma_{n_i}^{(i)}}\dots s_{c_1^{(i)},\Gamma_1^{(i)}}=s_{g_i,\Gamma_{g_i}}>1.
\end{equation*}
Similarly, for $\Gamma>\Gamma_{g_i}$, we have $s^{\prime}_{g_i,\Gamma}=s_{g_i,\Gamma}=1$. In this case $\Gamma^{\prime}_{g_i}=\Gamma_{g_i}$ and $g_i\succ 1$.

By choosing $\Gamma_0<\bigcap\limits_{i=1}^k\bigcap\limits_{j=1}^{n_i}\Gamma_j^{(i)}$ we get $g_1,\dots, g_k\succ1$ for any order on $S^{\prime}$ and any action $\beta_0$. 

By Corollary \ref{nosmallfree} there is a non-trivial element $h\in F_2\cap \bigcap\limits_{i=1}^k\bigcap\limits_{j=1}^{n_i}\Gamma_j^{(i)}$. With out loss of generality we may assume $h>1$. Let $h=c_m\dots c_1$.
Consider two cases:

Case 1: $\Gamma_h=\left(\Gamma_h\right)^a=\left(\Gamma_h\right)^b$.

In this case $\Gamma_h$ is a normal convex subgroup of $F_2$. We obtain the new order $\prec$ by reversing the signs of elements of $\Gamma_h$. In other words, the new order $\prec$ is generated by the positive cone
\begin{equation*}
    P_{\prec}=\left(P_<\setminus\Gamma_h\right)\cup\left(P_<^{-1}\cap\Gamma_h\right),
\end{equation*}
where $P_<$ is the positive cone of the order $<$.

Clearly, this doesn't affect the signs of $g_1,\dots, g_k$. Also, the new order $\prec$ is different from the old order $<$ since $\Gamma_h$ is nontrivial by Lemma \ref{nosmall}.




Case 2: $\Gamma_h\neq \left(\Gamma_h\right)^a$ or $\Gamma_h\neq \left(\Gamma_h\right)^b$. 

Consider a total left preorder $\leq_h$ given by $x\leq_h y$ when $(\Gamma_h)^x\leq (\Gamma_h)^y$. Recall that a preorder is a reflexive and transitive relation for which $x\leq y$ and $y\leq x$ may hold simultaneously. Note that the left preorder $\leq_h$ is completely determined by the action $\alpha_0$ and the convex subgroup $\Gamma_h$. Namely, $x\leq_h y$ if and only if $\alpha_0(x,\Gamma_h)=(\Gamma_h)^x\leq(\Gamma_h)^y=\alpha_0(y,\Gamma_h)$. The similarly defined left preorder $\preceq_h$ for an order $\prec$ described in Theorem \ref{gd} depends only on the action $\beta_0$ and the convex subgroup $\Gamma_h^{\prime}$. 

In order to change the order $<$, we change the induced left preorder $\leq_h$. First, we choose sufficiently small $\Gamma_0$ so that $\Gamma_h=\Gamma^{\prime}_h$. By the above discussion it is sufficient to take $\Gamma_0<\bigcap\limits_{j=1}^n\Gamma_h^{(j)}$. Then, we change the left preorder $\leq_h$ by changing the action $\alpha_0$. We will use the method similar to the argument for showing that free products don’t admit isolated left-orders \cite{Rivas}. For its adopted version for the free group $F_2$ see \cite[Theorem 10.15]{CR15}.

We construct sequences $d_1,d_2,\dots$ and $x_1,x_2,\dots$ as following:

Let $x_1=d_1\in\{a,a^{-1},b,b^{-1}\}$ be the minimal letter with respect to $\leq_h$. We choose $d_i\in\{a,a^{-1},b,b^{-1}\}$ to minimize $x_i\coloneqq d_i x_{i-1}$, $i=2,3,\dots$ with respect to $\leq_h$. Note that $x_i$ is a minimal (but not necessarily the smallest) word of length $i$ in $\leq_h$.
Equivalently, $x_i$ minimizes $\alpha_0(x_i,\Gamma_h)$. Note that because $\Gamma_h$ is not normal in $F_2$ we have $1>_h x_i>_h x_2>_h \dots$

Let $\Gamma_0=\left(\Gamma_{h}\right)^{x_m}$. Then, $\Gamma_0<\left(\Gamma_{h}\right)^x$, for any word $x$ of length less than $m$. In particular, $\Gamma_0< \Gamma_h^{(j)}$, $j=1,\dots,m$. Thus, for an order $\prec$ as in Theorem \ref{gd}, we have $\Gamma_h^{\prime}=\Gamma_h$. Also, since $\Gamma_0<\Gamma_h<\bigcap\limits_{i=1}^k\bigcap\limits_{j=1}^{n_i}\Gamma_j^{(i)}$,  $g_1,\dots, g_k\succ1$ holds.

Let $d=d_{m+1}$ and $d^{\prime}\in\{a,a^{-1},b,b^{-1}\}\setminus\{d,d^{-1}\}$ be such that $x_{m+1}=d x_m\leq_h d^{\prime}x_m\leq_h x_m$. Then, since $x_{m+1}<_hx_m$ we have
\begin{equation*}
    x_{m+1}<_h \max\{c^{-1}x_m,x_m\}, \;\; c=d,d^{\prime}.
\end{equation*}
In terms of the conjugation action, for $\Gamma=(\Gamma_h)^{x_{m+1}}$ this means
\begin{equation*}
    \Gamma<\Gamma_0\cup\left(\Gamma_0\right)^c, \;\; c=d,d^{\prime}.
\end{equation*}
Therefore, for the new action $\beta_0$, we may choose $\beta(d,\Gamma)$ and $\beta(d^{\prime},\Gamma)$ to be any sufficiently small (for instance, $<\Gamma$) convex subgroups.

We build $\beta_0$ so that $\beta_0(d_1,\Gamma_h)=\alpha_0(d_1,\Gamma_h)=(\Gamma_h)^{x_1}$,  $\beta_0(d_2,(\Gamma_h)^{x_1})=\alpha_0(d_2,(\Gamma_h)^{x_1})=(\Gamma_h)^{x_2}$, $\beta_0(d_3,(\Gamma_h)^{x_2})=\alpha_0(d_3,(\Gamma_h)^{x_2})=(\Gamma_h)^{x_3}$, \dots, $\beta_0(d_m,(\Gamma_h)^{x_{m-1}})=\alpha_0(d_m,(\Gamma_h)^{x_{m-1}})=(\Gamma_h)^{x_m}$, and $\beta_0(d,(\Gamma_h)^{x_m})>\beta_0(d^{\prime},(\Gamma_h)^{x_m})$.
In terms of the left preorder $\preceq_h$ the last means $dx_m\succ_h d^{\prime} x_m$ while in the left preorder $\leq_h$ we have $dx_m\leq_h d^{\prime} x_m$. Clearly, then the left preorders $\leq_h$ and $\preceq_h$ are different, and, therefore the orders $<$ and $\prec$ are different.

\begin{remark}
In terms of the order $<$ for a positive $h$ the condition $dx_m\leq_h d^{\prime} x_m$ means 
\begin{equation}\label{oinc}
    h^{dx_m}<\left(h^{d^{\prime}x_m}\right)^n
\end{equation}
for sufficiently large $n\in\N$. In the order $\prec$ the inequality \eqref{oinc} is reversed. 
\end{remark}

So we can take the order $\prec$ different from the order $<$, but still satisfying $g_1,\dots, g_k\succ1$. The order $<$ is not isolated in $O(F_2)$.

\end{proof}

\begin{corollary}
The space $O(F_2)$ is homeomorphic to the Cantor set.
\end{corollary}

\bibliographystyle{plain}
\end{document}